\documentclass[british]{article}
\usepackage[T1]{fontenc}
\usepackage[latin9]{inputenc}
\usepackage{geometry}
\usepackage{amsthm}
\usepackage{amsmath}
\usepackage{amssymb}
\usepackage{graphicx}
\usepackage{esint}

\makeatletter

\providecommand{\tabularnewline}{\\}

\numberwithin{equation}{section}
\numberwithin{figure}{section}
\newcommand{\lyxaddress}[1]{
\par {\raggedright #1
\vspace{1.4em}
\noindent\par}
}
  \theoremstyle{definition}
  \newtheorem{defn}{\protect\definitionname}
  \theoremstyle{plain}
  \newtheorem{thm}{\protect\theoremname}
 \ifx\proof\undefined\
   \newenvironment{proof}[1][\proofname]{\par
     \normalfont\topsep6\p@\@plus6\p@\relax
     \trivlist
     \itemindent\parindent
     \item[\hskip\labelsep
           \scshape
       #1]\ignorespaces
   }{%
     \endtrivlist\@endpefalse
   }
   \providecommand{\proofname}{Proof}
 \fi
  \theoremstyle{plain}
  \newtheorem{cor}{\protect\corollaryname}

\makeatother

\usepackage{babel}
\providecommand{\corollaryname}{Corollary}
\providecommand{\definitionname}{Definition}
\providecommand{\theoremname}{Theorem}

\begin{document}

\title{Matrix integrals and generating functions for enumerating rooted
hypermaps by vertices, edges and faces for a given number of darts}

\author{Jacob P Dyer}

\maketitle

\lyxaddress{jpd514@york.ac.uk}

\lyxaddress{Department of Mathematics, University of York, York YO10 5DD, UK}
\begin{abstract}
A recursive method is given for finding generating functions which
enumerate rooted hypermaps by number of vertices, edges and faces
for any given number of darts. It makes use of matrix-integral expressions
arising from the study of bipartite quantum systems. Direct evaluation
of these generating functions is then demonstrated through the enumeration
of all rooted hypermaps with up to 13 darts.\end{abstract}
\begin{description}
\item [{Keywords:}] enumeration, rooted hypermap, bipartite quantum system,
matrix integral, generating function, divergent power series
\end{description}

\section{Introduction}

This paper is an extension of work we carried out in a previous paper
\cite{Dyer2014a}. In that paper we showed how the mean value of traces
of integer powers of the reduced density operator of a finite-dimensional
bipartite quantum system are proportional to generating functions
for enumerating one-face rooted hypermaps. We then used this relation
to derive a matrix integral expression for these generating functions,
and found closed form expressions for them.

Matrix integral expressions derived from finding the average of a
function of the reduced density operator have been studied for some
time \cite{Lloyd1988,Page1993,Foong1994,Sanchez-Ruiz1995,Sen1996},
so numerous methods for evaluating them have been described. In particular,
Lloyd and Pagels were able to reduce the matrix integral to an integral
over the space of eigenvalues with the density function \cite{Lloyd1988}
\[
P(p_{1},\ldots,p_{m})dp_{1}\ldots dp_{m}\propto\delta\left(1-\sum_{i=1}^{m}p_{1}\right)\Delta^{2}(p_{1},\ldots,p_{m})\prod_{k=1}^{m}p_{k}^{n-m}dp_{k}
\]
where $\Delta(p_{1},\ldots,p_{m})$ is the Vandermonde determinant
of the eigenvalues of the reduced density operator. Using this, in
conjunction with the work of Sen \cite{Sen1996}, we were able to
evaluate a closed-form expression for the one-face generating functions.

In this paper we extend these methods to derive expressions for generating
functions which enumerate all rooted hypermaps by number of vertices,
edges and faces for a given number of darts (we give a definition
of rooted hypermaps in Section \ref{sec:Representing-rooted-hypermaps}).
These generating functions are defined recursively in terms of another
expression called $F(m,n,\lambda;x)$, which we define in Section
\ref{sec:Additional-generating-functions} and is itself evaluated
using a matrix integral as above.

Previous work already exists on enumeration of hypermaps \cite{Arques1987,Mednykh2010,Walsh2012},
and in particular Walsh managed to enumerate all rooted hypermaps
with up to 12 darts by number of edges, vertices and darts, and genus
\cite{Walsh2012}. But as far as we are aware this is the first time
that generating functions for enumerating all rooted hypermaps by
these properties have been found without direct computation of the
hypermaps themselves. By avoiding having to generate the hypermaps
individually, we are able to vastly speed up the process of enumeration
(there are more than $r!$ hypermaps with $r$ darts \cite{Dyer2014a},
so generating them all is a very slow process).

We will give an overview of our previous work in Section \ref{sub:One-face-hypermaps},
before showing how best to generalise it to multiple faces in Section
\ref{sub:Multiple-faces}. We will then use this to study the global
generating function for rooted hypermaps $H(m,n,\lambda;x)$ in Sections
\ref{sec:Additional-generating-functions} and \ref{sec:Hypermap-generating-functions},
before looking at the process of evaluating these functions and extracting
hypermap counts from them in Section \ref{sec:Evaluating}.

\section{Representing rooted hypermaps\label{sec:Representing-rooted-hypermaps}}

A thorough discussion of hypermaps can be found in \cite{Lando2004}.

A \emph{hypermap} is a generalisation of a map (a graph embedded on
an orientable surface so that its complement consists only of regions
which are homeomorphic to the unit disc) in which the edges are capable
of having any positive number of connections to vertices instead of
the usual two. Hypermaps can be thought of as equivalent to bipartite
bicoloured maps on the same surface (with the two colours of vertices
in the map representing the vertices and edges of the hypermap) \cite{Walsh1975}.
Each edge-vertex connection (the edges in the eqivalent bipartite
map) is called a \emph{dart}, and a \emph{rooted hypermap} is a hypermap
where one of the darts has been labelled as the \emph{root}, making
it distinct from the others.

The embedding of a hypergraph (the analogue of a graph) on an orientable
surface to produce a hypermap can be represented in other ways which
do not require explicit consideration of the surface involved. These
are called \emph{combinatorial embeddings}, and one such method uses
an object called a 3-constellation:

\bigskip{}

\begin{defn}
A 3-constellation is an ordered triple $\{\xi,\eta,\chi\}$ of permutations
acting on some set $R$, satisfying the following two properties:
\begin{enumerate}
\item The group generated by $\{\xi,\eta,\chi\}$ acts transitively on $R$.
\item The product $\xi\eta\chi$ equals the identity.
\end{enumerate}
\end{defn}
\bigskip{}

A hypermap $H$ with $r$ darts can be expressed using a 3-constellation
on the set $R=[1\ldots r]$. If the elements of $R$ are associated
with the darts in $H$, then the actions of $\xi$, $\eta$ and $\chi$
are to cycle the darts around their adjacent faces, edges and vertices
respectively. For our purposes here, the important result is that
the number of faces, edges and vertices in a hypermap $H\equiv\{\xi,\eta,\chi\}$
are given by the number of cycles in $\xi$, $\eta$ and $\chi$ respectively
\cite[p 43]{Lando2004}.

Two 3-constellations $\{\xi,\eta,\chi\}$ and $\{\xi',\eta',\chi'\}$
are isomorphic to each other if they are related by the bijection
\begin{equation}
\tau\,:\,\{\xi,\chi,\eta\}\rightarrow\{\xi',\chi',\eta'\}=\{\tau\xi\tau^{-1},\tau\eta\tau^{-1},\tau\chi\tau^{-1}\}\label{eq:tau-mapping}
\end{equation}
for some permutation $\tau$ \cite[p 8]{Lando2004}. are isomorphic
to each other (the action of $\tau$ on the hypermap as given above
simply involves a reordering of the darts in the set $R$ without
changing the connectivity). With this representation of hypermap isomorphism
established, we define rooted hypermaps as hypermaps with the aditional
property that they are only equivalent under the action of $\tau$
only when $\tau(1)=1$ (i.e. choosing for the root dart to have the
label $1$).

\subsection{One-face hypermaps\label{sub:One-face-hypermaps}}

\begin{figure}
\hfill{}\includegraphics[width=10cm]{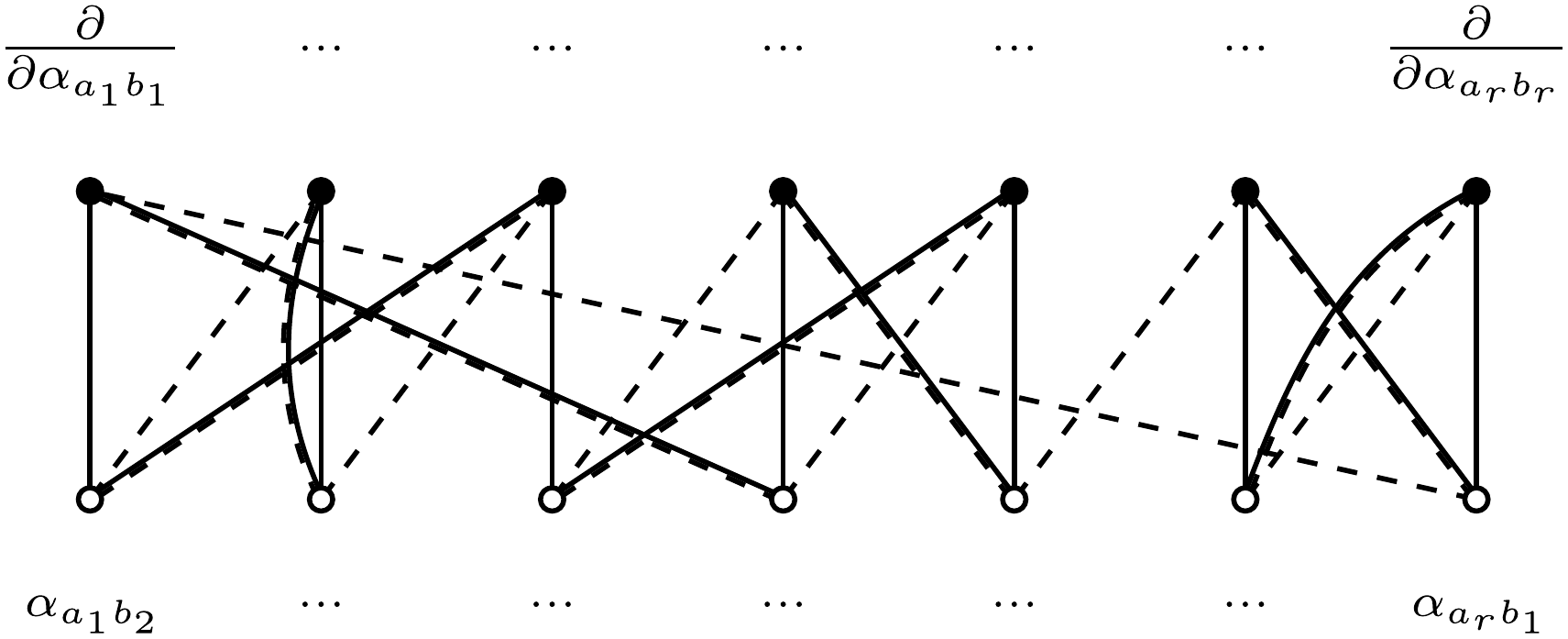}\hfill{}

\protect\caption{\label{fig:one-face-diag}A diagrammatic representation of a one-face
rooted hypermap $H\equiv\{\xi,\eta,\chi\}$, with $\xi=(12\ldots r)$
and $\eta=(1453)(2)(67)$, referred to as a \emph{ladder diagram}.
Mapping from black to white, the single dashed lines represent $\xi^{-1}$
and the double (solid and dashed) lines represent $\eta$. The number
of edges in the hypermap equals the number of closed solid loops,
while the number of vertices equals the number of closed dashed loops
(the double lines count as either solid or dashed). Also shown are
the correspondence between the nodes and the terms in (\ref{eq:P_r}),
which is used in the evaluation of $P_{r}(m,n)$.}
\end{figure}

In our previous paper paper, we used the 3-constellation representation
of hypermap embedding to define a diagrammatic representation of rooted
hypermaps \cite{Dyer2014a} (see Figure \ref{fig:one-face-diag}).
If we define
\[
\xi=(12\ldots r),
\]
then the set of rooted hypermaps with one face is equivalent to the
set of permutations $\eta$ on $[1\ldots r${]} (i.e. the symmetric
group $Sym_{r}$) through the then any rooted hypermap with one face
is equivalent to , then the set of nonisomorphic rooted hypermaps
with $r$ darts is equivalent to the set of permutations all $\eta$
on $[1\ldots r]$ (i.e. the symmetric group $Sym_{r}$) through the
bijection
\[
\eta\rightarrow H_{\eta}\equiv\{\xi,\eta,\eta^{-1}\xi^{-1}\}
\]
(as $\xi$ is fixed, no two choices of $\eta$ will result in equivalent
rooted hypermaps). The diagrammatic representation in Figure \ref{fig:one-face-diag}
(here referred to as a ladder diagram) allows us to quickly count
the number of vertices and edges in $H_{\eta}$ by counting closed
loops (the numbers of which are equal to the number permutations in
$\eta$ and $\xi\eta$).

We showed that these diagrams also arise in the evaluation of the
function

\begin{equation}
P_{r}(m,n)=\left.\frac{\partial}{\partial\alpha_{a_{1}b_{1}}}\ldots\frac{\partial}{\partial\alpha_{a_{r}b_{r}}}(\alpha_{a_{1}b_{2}}\ldots\alpha_{a_{r}b_{1}})\right|_{\alpha=0},\label{eq:P_r}
\end{equation}
where $\alpha$ is an $m\times n$ real matrix: when the multiderivative
in (\ref{eq:P_r}) is fully expanded out, it has the form
\[
P_{r}(m,n)=\sum_{\eta\in Sym_{r}}\prod_{i=1}^{r}\delta[a_{i},a_{\eta(i)}]\delta[b_{i},b_{\xi\eta(i)}]=\sum_{\eta\in Sym_{r}}m^{cyc(\eta)}n^{cyc(\xi\eta)}
\]
where $cyc(\sigma)$ is the number of cycles in the permutation $\sigma$.
As $cyc(\eta)$ and $cyc(\xi\eta)$ are respectively the number of
edges and vertices in the rooted hypermap $H_{\eta}\equiv\{\xi,\eta,\eta^{-1}\xi^{-1}\}$,
and the number of faces in $H_{\eta}$ is $cyc(\xi)=1$, $P_{r}$
is therefore the generating function for enumerating one-face rooted
hypermaps with $r$ darts by number of edges and vertices. $P_{r}$
can also be computed using ladder diagrams as above, then, each diagram
contributing a single $m^{cyc(\eta)}n^{cyc(\xi\eta)}$ term.

We also showed, through Gaussian integration, that
\begin{eqnarray}
P_{r}(m,n) & = & \int_{\mathbb{C}^{mn}}d^{2mn}xe^{-x\cdot x}x_{a_{1}b_{1}}x_{a_{1}b_{2}}^{*}\ldots x_{a_{r}b_{r}}x_{a_{r}b_{1}}^{*}\nonumber \\
 & = & \frac{\Gamma(mn+r)}{\Gamma(mn)}\langle\text{Tr}[(\hat{\rho}^{A})^{r}]\rangle,\label{eq:P_r-quantum}
\end{eqnarray}
where $\hat{\rho}^{A}$ is the reduced density operator of an $m$-dimensional
subsystem of an $mn$-dimensional bipartite quantum system, and the
mean is being taken over all possible pure states of the overall bipartite
system. What this means is explained in more detail in \cite{Dyer2014a},
but the facts of most relevance here are that (\ref{eq:P_r-quantum})
is symmetric in $m$ and $n$, and, when $n\ge m$, the mean can be
represented as an integral over the eigenvalues $(p_{1},\ldots,p_{m})$
of $\hat{\rho}^{A}$ with the density function \cite{Lloyd1988,Dyer2014}
\[
P(p_{1},\ldots,p_{m})dp_{1}\ldots dp_{m}\propto\delta\left(1-\sum_{i=1}^{m}p_{1}\right)\Delta^{2}(p_{1},\ldots,p_{m})\prod_{k=1}^{m}p_{k}^{n-m}dp_{k},
\]
where $\Delta^{2}(p_{1},\ldots,p_{m})$ is the Vandermonde discriminant
of the eigenvalues, giving
\[
\langle\text{Tr}[(\hat{\rho}^{A})^{r}]\rangle\propto\int\delta\left(1-\sum_{i=1}^{m}p_{1}\right)\Delta^{2}(p_{1},\ldots,p_{m})\prod_{k=1}^{m}(p_{k}^{n-m}dp_{k})\sum_{j=1}^{m}p_{j}^{r}.
\]
Using a co-ordinate substitution given in \cite{Page1993}, we multiply
this by the factor
\[
\frac{1}{\Gamma(mn+r)}\int_{0}^{\infty}\lambda^{mn+r-1}e^{-\lambda}d\lambda
\]
and define $q_{i}=\lambda p_{i}$, integrating over $\lambda$ in
order to remove the $\delta$ function, giving
\[
\langle\text{Tr}[(\hat{\rho}^{A})^{r}]\rangle\propto\frac{1}{\Gamma(mn+r)}\int\Delta^{2}(q_{1},\ldots,q_{m})\prod_{k=1}^{m}(e^{-q_{k}}q_{k}^{n-m}dq_{k})\sum_{j=1}^{m}q_{j}^{r}.
\]
Finally, we normalise this by using the fact that, as $n\ge m$, $\langle\text{Tr}[(\hat{\rho}^{A})^{0}]\rangle=\langle\text{Tr}[I_{m}]\rangle=m$,
giving
\[
\langle\text{Tr}[(\hat{\rho}^{A})^{r}]\rangle=\frac{\Gamma(mn)}{\Lambda_{mn}\Gamma(mn+r)}\int\Delta^{2}(q_{1},\ldots,q_{m})\prod_{k=1}^{m}(e^{-q_{k}}q_{k}^{n-m}dq_{k})\sum_{j=1}^{m}q_{j}^{r}
\]
where
\[
\Lambda_{mn}=\int\Delta^{2}(q_{1},\ldots,q_{m})\prod_{k=1}^{m}(e^{-q_{k}}q_{k}^{n-m}dq_{k}).
\]

We now need to generalise both the concept of ladder diagrams and
the closely connected $P_{r}(m,n)$ functions in order to enumerate
hypermaps with more than one face, and we will define these generalisations
in the next section.

\subsection{Multiple faces\label{sub:Multiple-faces}}

\begin{figure}
\hfill{}\includegraphics[width=9cm]{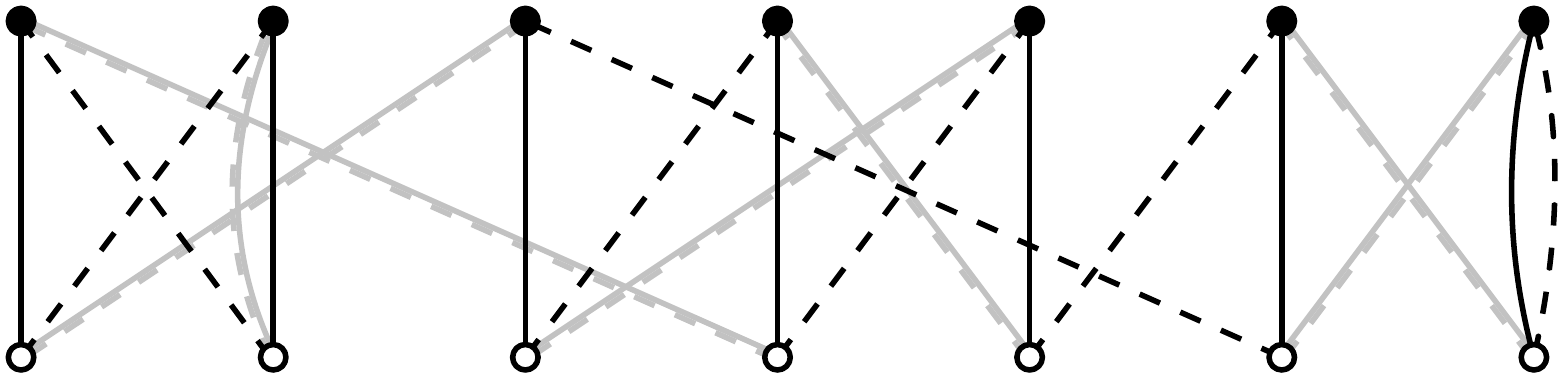}\hfill{}

\protect\caption{\label{fig:three-face-diag}A ladder diagram of a three-face rooted
hypermap $H\equiv\{\xi,\eta,\chi\}$. As in Figure \ref{fig:one-face-diag},
$\eta=(1453)(2)(67)$, but in this case $\xi=(12)(3456)(7)$. The
double edges have been greyed out for clarity. Summation over possible
permutations $\eta$ in this case generates $P_{2,4,1}(m,n)$}
\end{figure}

In Figure \ref{fig:one-face-diag}, there was just one loop consisting
only of single lines (i.e. single solid lines and single dotted lines,
but not the solid/dotted paired lines), which corresponded to the
single cycle in $\xi$, and therefore to the face in the associated
rooted hypermap. It follows that a hypermap with multiple faces would
have a diagram with multiple such loops (i.e. $\xi$ has multiple
cycles, one for each face). An example of such a diagram is shown
in Figure \ref{fig:three-face-diag}.

In these diagrams, the solid lines in combination with the dotted
lines defined by $\xi$ can be thought of as a fixed backbone, on
which the double lines given by $\eta$ are superimposed. We described
in Section \ref{sub:One-face-hypermaps} how, when $\xi=(12\ldots r)$,
we can sum over all possible $\eta$ and in each case count the solid
and dotted loops in order to get a generating function for enumerating
rooted one-face hypermaps with $r$ darts. We also showed that this
function was equivalent to (\ref{eq:P_r}) and (\ref{eq:P_r-quantum}).

We can apply the same procedure to diagrams with other backbones.
Looking at (\ref{eq:P_r}) and (\ref{eq:P_r-quantum}), we can see
that the single cycle of length $r$ in $\xi$ corresponds to a term
$\text{Tr}[(\hat{\rho}^{A})^{r}]$ in the quantum expression for $P_{r}$.
By extension it follows that if consists of $N$ cycles with lengths
$r_{1},r_{2},\ldots,r_{N}$ (e.g. the $\xi$ used in Figure (\ref{fig:three-face-diag})
corresponds to $N=3$, $\{r_{1},r_{2},r_{3}\}=\{2,4,1\}$), summing
over all ladder diagrams with such a backbone and following the same
procedure as in section \ref{sub:One-face-hypermaps}, we get the
function
\begin{eqnarray}
P_{r_{1}r_{2}\ldots r_{N}}(m,n) & = & \frac{\Gamma(mn+\Sigma_{i=1}^{N}r_{i})}{\Gamma(mn)}\left\langle \prod_{j=1}^{N}\text{Tr}[(\hat{\rho}^{A})^{r_{j}}]\right\rangle \nonumber \\
 & = & \frac{1}{\Lambda_{mn}}\int\Delta^{2}(q_{1},\ldots,q_{m})\prod_{k=1}^{m}(e^{-q_{k}}q_{k}^{n-m}dq_{k})\nonumber \\
 &  & \qquad\times\prod_{i=1}^{N}\sum_{j=1}^{m}q_{j}^{r_{i}},\label{eq:P_mult}
\end{eqnarray}
again valid when $m\le n$.

These functions are not yet useful generating functions, however,
for two reasons: the sum over diagrams used to calculate them can
include disconnected diagrams (hypermaps are necessarily connected,
so cannot correspond to disconnected ladder diagrams), and any two
hypermaps which are related through cyclic permutation of one of the
cycles in $\xi$ are equivalent, producing a degeneracy. We will overcome
these issues in the following sections; first we will define some
additional functions in terms of the various $P_{r_{1}\ldots}$ in
Section \ref{sec:Additional-generating-functions} which will account
for the presence of disconnected diagrams, and then we will use these
to construct global generating functions for counting rooted hypermaps
in Section \ref{sec:Hypermap-generating-functions}.

\section{Connected diagrams\label{sec:Additional-generating-functions}}

As stated in the previous section, the functions $P_{r_{1}\ldots}$
defined in (\ref{eq:P_mult}) are generating functions each of which
count over a set of ladder diagrams. As defined, however, they include
disconnected diagrams in this count, whereas we require generating
functions which count only over connected diagrams. In this section
we will define such functions.

For any given $P_{r_{1}\ldots r_{N}}$, define $\bar{P}_{r_{1}\ldots r_{N}}$
to be a generating function defined as a summation over the same set
of diagrams as $P_{r_{1}\ldots}$ except with any disconnected diagrams
excluded. In the one-loop case, $P_{r}=\bar{P}_{r}$ as all one-loop
ladder diagrams are connected. When there is more than one loop present,
$P_{r_{1}\ldots}$ may be factorised in terms of $\bar{P}_{r_{1}\ldots}$
using the fact that any disconnected ladder diagram can be split into
a number of disjoint connected subdiagrams. We write this factorisation
\begin{equation}
P_{rr_{1}\ldots r_{N}}=\bar{P}_{rr_{1}\ldots r_{N}}+\bar{P}_{r}P_{r_{1}\ldots r_{N}}+\sum_{u\cup v=\{r_{1}\ldots r_{N}\}}^{u,v\ne\emptyset}\bar{P}_{ru_{1}\ldots}P_{v_{1}\ldots},\label{eq:Pbar-full}
\end{equation}
where the summation is over all partitions of the ordered multiset
$\{r_{1}\ldots r_{N}\}$ into two disjoint non-empty subfamilies.
This factorisation works by breaking up each diagram in the summation
into its disjoint connected subdiagrams and considering which subdiagram
the loop of length $r$ is in. This loop is factored out in a $\bar{P}$
term. As an example, when $N=3$,
\begin{eqnarray*}
P_{rabc} & = & \bar{P}_{rabc}+\bar{P}_{r}P_{abc}+\bar{P}_{ra}P_{bc}+\bar{P}_{rb}P_{ac}+\bar{P}_{rc}P_{ab}\\
 &  & +\bar{P}_{rab}P_{c}+\bar{P}_{rac}P_{b}+\bar{P}_{rbc}P_{a}.
\end{eqnarray*}
If the definition of $P_{r\ldots}$ is extended to include 
\[
P(m,n)=1,
\]
which is consistent with (\ref{eq:P_mult}), then (\ref{eq:Pbar-full})
can be written more simply as
\begin{equation}
P_{rr_{1}\ldots r_{N}}=\sum_{u\cup v=\{r_{1}\ldots r_{N}\}}\bar{P}_{ru_{1}\ldots}P_{v_{1}\ldots},\label{eq:Pbar-simplified}
\end{equation}
where $u$ and $v$ are now allowed to be empty.

(\ref{eq:Pbar-full}) and (\ref{eq:P_mult}) can be used recusively
to construct integral expressions for any given $\bar{P}_{r\ldots}$.
However, when constructing the global generating function in Section
\ref{sec:Hypermap-generating-functions}, it will be more useful to
work with the functions 
\begin{eqnarray}
\Pi_{r}^{(N)}(m,n;x) & = & \sum_{r_{1}=1}^{\infty}\frac{x^{r_{1}}}{r_{1}}\ldots\sum_{r_{N}=1}^{\infty}\frac{x^{r_{N}}}{r_{N}}P_{rr_{1}\ldots r_{N}}(m,n)\label{eq:Pi}\\
\bar{\Pi}_{r}^{(N)}(m,n;x) & = & \sum_{r_{1}=1}^{\infty}\frac{x^{r_{1}}}{r_{1}}\ldots\sum_{r_{N}=1}^{\infty}\frac{x^{r_{N}}}{r_{N}}\bar{P}_{rr_{1}\ldots r_{N}}(m,n)\label{eq:Pi_bar}\\
\Sigma^{(N)}(m,n;x) & = & \sum_{r_{1}=1}^{\infty}\frac{x^{r_{1}}}{r_{1}}\ldots\sum_{r_{N}=1}^{\infty}\frac{x^{r_{N}}}{r_{N}}P_{r_{1}\ldots r_{N}}(m,n).\label{eq:Sigma}
\end{eqnarray}
These are specifically defined as formal power series in $x$; in
general these series will be divergent if treated as functions of
a finite parameter $x$. Noting the special cases $\Pi_{r}^{(0)}(m,n;x)=\bar{\Pi}_{r}^{(0)}(m,n;x)=P_{r}(m,n)$
and $\Sigma^{(0)}(m,n;x)=1$, we use (\ref{eq:Pbar-simplified}) to
derive the recursion relation
\begin{eqnarray}
\Pi_{r}^{(N)}(m,n;x) & = & \sum_{r_{1}=1}^{\infty}\frac{x^{r_{1}}}{r_{1}}\ldots\sum_{r_{N}=1}^{\infty}\frac{x^{r_{N}}}{r_{N}}P_{rr_{1}\ldots r_{N}}(m,n)\nonumber \\
 & = & \sum_{r_{1}=1}^{\infty}\frac{x^{r_{1}}}{r_{1}}\ldots\sum_{r_{N}=1}^{\infty}\frac{x^{r_{N}}}{r_{N}}\sum_{u\cup v=\{r_{1}\ldots r_{N}\}}\bar{P}_{ru_{1}\ldots}(m,n)P_{v_{1}\ldots}(m,n)\nonumber \\
 & = & \sum_{k=0}^{N}\binom{N}{k}\bar{\Pi}_{r}^{(k)}(m,n;x)\Sigma^{(N-k)}(m,n;x),\label{eq:Pi-recurse}
\end{eqnarray}
with the sum over partitions in (\ref{eq:Pbar-simplified}) becoming
a sum over the different possible sizes of the partitions instead.

In addition to these three sets of series, we will need one more series
to be defined:
\begin{equation}
F(m,n,\lambda;x)=\sum_{N=0}^{\infty}\frac{\lambda^{N}}{N!}\Sigma^{(N)}(m,n;x).\label{eq:F}
\end{equation}
This series' derivative satisfies
\begin{eqnarray}
x\frac{\partial}{\partial x}F(m,n,\lambda;x) & = & x\frac{\partial}{\partial x}\sum_{N=0}^{\infty}\frac{\lambda^{N}}{N!}\sum_{r_{1}=1}^{\infty}\frac{x^{r_{1}}}{r_{1}}\ldots\sum_{r_{N}=1}^{\infty}\frac{x^{r_{N}}}{r_{N}}P_{r_{1}\ldots r_{N}}(m,n)\nonumber \\
 & = & \sum_{N=0}^{\infty}\frac{\lambda^{N}}{N!}\cdot N\sum_{r_{1}=1}^{\infty}x^{r_{1}}\sum_{r_{2}=1}^{\infty}\frac{x^{r_{2}}}{r_{2}}\ldots\sum_{r_{N}}^{\infty}\frac{x^{r_{N}}}{r_{N}}P_{r_{1}\ldots r_{N}}(m,n)\nonumber \\
 & = & \sum_{N=1}^{\infty}\frac{\lambda^{N}}{(N-1)!}\sum_{r=1}^{\infty}x^{r}\Pi_{r}^{(N-1)}(m,n;x)\nonumber \\
 & = & \sum_{N=0}^{\infty}\frac{\lambda^{N+1}}{N!}\sum_{r=1}^{\infty}x^{r}\Pi_{r}^{(N)}(m,n;x),\label{eq:F-derivative}
\end{eqnarray}
and when $x$ is set to zero,
\begin{eqnarray}
F(m,n,\lambda;0) & = & \sum_{N=0}^{\infty}\frac{\lambda^{N}}{N!}\Sigma^{(N)}(m,n;0)\nonumber \\
 & = & \Sigma^{(0)}(m,n;0)\nonumber \\
 & = & 1.\label{eq:F_0}
\end{eqnarray}

With these various functions and series defined, we can proceed to
define the global generating function for enumerating rooted hypermaps
in terms $F$. After doing this in the next section, we will return
to $F$ in Section \ref{sec:Evaluating} and discuss methods for evaluating
it.

\section{Hypermap generating functions\label{sec:Hypermap-generating-functions}}

Let us define $H(m,n,\lambda;x)$ as the generating function for enumerating
all rooted hypermaps in the form
\begin{equation}
H(m,n,\lambda;x)=\sum_{e,v,f,r}H_{vefr}m^{v}n^{e}\lambda^{f}x^{r},\label{eq:counting-basic}
\end{equation}
where $H_{vefr}$ is the number of rooted hypermaps with v edges,
$e$ vertices, $f$ faces and $r$ darts. As with the expressions
used in Section \ref{sec:Additional-generating-functions}, this generating
function is strictly speaking a formal power series in $x$ which
will be dievergent in general. However, if we write
\[
H(m,n,\lambda;x)=\sum_{r=0}^{\infty}H_{r}(m,n,\lambda)x^{r},
\]
then the individual $H_{r}$ will be well-behaved polynomial functions
enumerating all rooted hypermaps with $r$ darts. Ultimately our aim
will be to compute these.

It is worth noting the symmetry properties of these functions:

\bigskip{}

\begin{thm}
\label{thm:symmetric}Each $H_{r}$ is completely symmetric in its
three parameters, or, equivalently, 
\[
H_{r}(m,n,\lambda)=H_{r}(n,m,\lambda)=H_{r}(m,\lambda,n).
\]
\end{thm}
\begin{proof}
This result follows easily from considering a rooted hypermap as a
3-constellation $\{\xi,\chi,\eta\}$. The mapping
\[
T_{ef}\,:\,\{\xi,\chi,\eta\}\rightarrow\{\chi^{-1},\xi^{-1},\eta^{-1}\}
\]
maps rooted hypermaps with $r$ darts onto each other, and specifically
maps a rooted hypermap with $v$ vertices, $e$ edges and $f$ faces
onto one with $v$ vertices, $f$ faces and $e$ edges. As $T_{ef}$
is bijective (it is its own inverse), this means that $H_{vefr}=H_{vfer}$,
and so
\[
H_{r}(m,n,\lambda)=\sum_{v,e,f}H_{vefr}m^{v}n^{e}\lambda^{f}=H_{r}(m,\lambda,n).
\]

Similarly, the mapping
\[
T_{ve}\,:\,\{\xi,\chi,\eta\}\rightarrow\{\xi^{-1},\eta^{-1},\chi^{-1}\}
\]
is a bijection which swaps the number of edges and vertices in each
rooted hypermaps, meaning $H_{vefr}=H_{evfr}$ and 
\[
H_{r}(m,n,\lambda)=H_{r}(n,m,\lambda).
\]

\end{proof}
\bigskip{}

\begin{figure}
\hfill{}\includegraphics[width=1\textwidth]{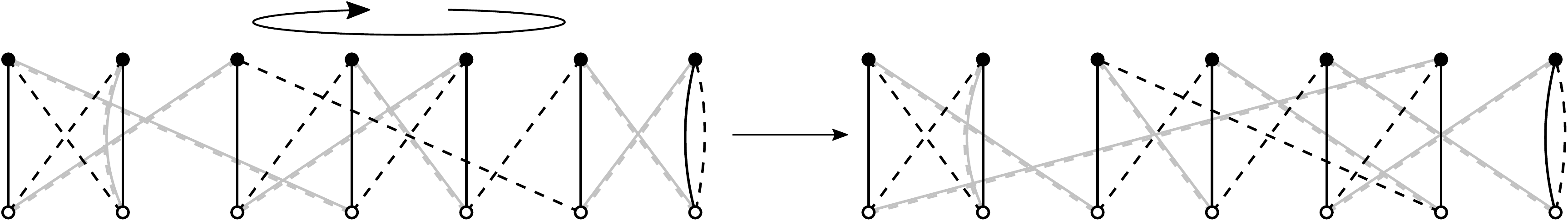}\hfill{}

\protect\caption{\label{fig:A-cyclic-permutation}A cyclic permutation one place to
the left applied to the second loop in a ladder diagram. As the two
diagrams are isomorphic they are equivalent to the same hypermap,
but they contribute separately to the function $\bar{P}_{241}$ as
the diagrams themselves are distinct. The total degeneracy in this
case is $r_{2}\cdot r_{3}=4$ (there is no degeneracy associated with
the first loop as it contains the root, and is therefore fixed against
permutation).}

\end{figure}

While we cannot evaluate $H$ directly, we are able to define it in
relation to the series $F$ defined previously:

\bigskip{}

\begin{thm}
\label{thm:P-given-F}The generating function $H$ satisfies the relation
\[
H(m,n,\lambda;x)F(m,n,\lambda;x)=x\frac{\partial}{\partial x}F(m,n,\lambda;x).
\]
\end{thm}
\begin{proof}
Each $\bar{P}_{r_{1}\ldots r_{N}}$ is a generating function for a
set of rooted hypermaps with $N$ faces (one for each of the loops
in the associated ladder diagrams), and if all possible $\bar{P}_{r_{1}\ldots r_{N}}$
for fixed $r_{1}+\ldots+r_{N}=r$ are summed over, then the resulting
function will include terms for every rooted hypermap with $N$ faces
and $r$ darts, as any such hypermap has at least one associated ladder
diagram which contributes a term to one of the $\bar{P}_{r_{1}\ldots r_{N}}$.
However, each such hypermap with have a total of $(N-1)!r_{2}r_{3}\ldots r_{N}$
such ladder diagrams (the $N-1$ loops of length $r_{2}$ through
$r_{N}$ can be put in any order to get distinct diagrams to get a
degeneracy of $(N-1)!$, and each of these loops can have its nodes
permuted cyclically -- see Figure \ref{fig:A-cyclic-permutation}
-- giving a degeneracy of $r_{2}r_{3}\ldots r_{N}$; in both cases
the $r_{1}$ loop is fixed because it is associated with the root),
so in order to get a generating function which only counts each rooted
hypermap once, each $\bar{P}_{r_{1}\ldots r_{N}}$ must be divided
by this degeneracy.

We therefore write out $H$ explicitly by summing over all $P_{r_{1}\ldots r_{N}}$,
dividing each by $(N-1)!r_{2}r_{3}\ldots r_{N}$, and multiplying
each by $\lambda^{N}x^{r_{1}+\ldots+r_{N}}$ in order to index the
enumeration by number of faces and darts as well. The resulting expression
is

\begin{eqnarray*}
H(m,n,\lambda;x) & = & \sum_{N=1}^{\infty}\frac{\lambda^{N}}{(N-1)!}\sum_{r_{1}=1}^{\infty}x^{r_{1}}\sum_{r_{2}=1}^{\infty}\frac{x^{r_{2}}}{r_{2}!}\ldots\sum_{r_{N}=1}^{\infty}\frac{x^{r_{N}}}{r_{N}}\bar{P}_{r_{1}r_{2}\ldots r_{N}}(m,n)\\
 & = & \sum_{N=0}^{\infty}\frac{\lambda^{N+1}}{N!}\sum_{r=1}^{\infty}x^{r}\sum_{r_{1}=1}^{\infty}\frac{x^{r_{1}}}{r_{1}!}\ldots\sum_{r_{N}=1}^{\infty}\frac{x^{r_{N}}}{r_{N}}\bar{P}_{rr_{1}\ldots r_{N}}(m,n).
\end{eqnarray*}
We simplify this by substituting in (\ref{eq:Pi_bar}):
\begin{equation}
H(m,n,\lambda;x)=\sum_{N=0}^{\infty}\frac{\lambda^{N+1}}{N!}\sum_{r=1}^{\infty}x^{r}\bar{\Pi}_{r}^{(N)}(m,n;x).\label{eq:P-by-Pi}
\end{equation}
Now, mutliplying this by $F$ as defined in (\ref{eq:F}), we get
\begin{eqnarray*}
H(m,n,\lambda;x)F(m,n,\lambda;x) & = & \sum_{N=0}^{\infty}\frac{\lambda^{N+1}}{N!}\sum_{r=1}^{\infty}x^{r}\bar{\Pi}_{r}^{(N)}(m,n;x)\sum_{k=0}^{\infty}\frac{\lambda^{k}}{k!}\Sigma^{(k)}(m,n;x)\\
 & = & \sum_{r=1}^{\infty}x^{r}\sum_{k=0}^{\infty}\sum_{N=0}^{\infty}\frac{\lambda^{N+k+1}}{N!k!}\bar{\Pi}_{r}^{(N)}(m,n;x)\Sigma^{(k)}(m,n;x)\\
 & = & \sum_{r=1}^{\infty}x^{r}\sum_{k=0}^{\infty}\sum_{N=k}^{\infty}\frac{\lambda^{N+1}}{(N-k)!k!}\bar{\Pi}_{r}^{(N-k)}(m,n;x)\Sigma^{(k)}(m,n;x)\\
 & = & \sum_{r=1}^{\infty}x^{r}\sum_{N=0}^{\infty}\sum_{k=0}^{N}\frac{\lambda^{N+1}}{(N-k)!k!}\bar{\Pi}_{r}^{(N-k)}(m,n;x)\Sigma^{(k)}(m,n;x)\\
 & = & \sum_{N=0}^{\infty}\frac{\lambda^{N+1}}{N!}\sum_{r=1}^{\infty}x^{r}\sum_{k=0}^{N}\binom{N}{k}\bar{\Pi}_{r}^{(N-k)}(m,n;x)\Sigma^{(k)}(m,n;x).
\end{eqnarray*}
This has a clear similarity to (\ref{eq:Pi-recurse}), so we substitute
in (\ref{eq:Pi-recurse}) and (\ref{eq:F-derivative}), giving
\begin{eqnarray}
H(m,n,\lambda;x)F(m,n,\lambda;x) & = & \sum_{N=0}^{\infty}\frac{\lambda^{N+1}}{N!}\sum_{r=1}^{\infty}x^{r}\Pi_{r}^{(N)}(m,n;x)\nonumber \\
 & = & x\frac{\partial}{\partial x}F(m,n,\lambda;x).\label{eq:P-F-relation}
\end{eqnarray}

\end{proof}
\bigskip{}

If $F$ and $H$ were both well-behaved functions, this expression
would be sufficient to evaluate $H$ given $F$. As both are formal
power series, however, it is only meaningful to consider this expression
in terms of the terms in these series. Defining the functions $F_{r}(m,n,\lambda)$
such that
\[
F(m,n,\lambda;x)=\sum_{r=0}^{\infty}F_{r}(m,n,\lambda)x^{r},
\]
(\ref{eq:P-F-relation}) becomes
\begin{equation}
\sum_{k=0}^{r}H_{r-k}(m,n,\lambda)F_{k}(m,n,\lambda)=rF_{r}(m,n,\lambda).\label{eq:generate-recurse}
\end{equation}
When $r=0$ this simply gives
\[
H_{0}(m,n,\lambda)=0,
\]
and then we can recursively construct other $H_{r}$ for $r>0$. For
example, the first few are
\begin{eqnarray*}
H_{1}(m,n,\lambda) & = & F_{1}(m,n,\lambda)\\
H_{2}(m,n,\lambda) & = & 2F_{2}(m,n,\lambda)-[F_{1}(m,n,\lambda)]^{2}\\
H_{3}(m,n,\lambda) & = & 3F_{3}(m,n,\lambda)-3F_{1}(m,n,\lambda)F_{2}(m,n,\lambda)+[F_{1}(m,n,\lambda)]^{2},
\end{eqnarray*}
where we have made use of the fact that $F_{0}(m,n,\lambda)=F(m,n,\lambda;0)=0$
as shown in (\ref{eq:F_0}).

All that remains, then, is to evaluate the various $F_{r}$. We will
do this in the next section.

\section{Evaluating $F_{r}$\label{sec:Evaluating}}

We now have the generating function $H$ defined in terms of the series
$F$. The problem of evaluating terms in the $x$-series expansion
of $H$ is therefore equivalent to the problem of evaluating the the
terms in $F$. In this section we will establish an integral representation
of $F$ and then discuss the use of this to explicitly evaluate the
terms $F_{r}$ in $F$.
\begin{thm}
\label{thm:F-integral}The series $F$ has the integral representation
\begin{equation}
F(m,n,\lambda;x)=\frac{1}{\Lambda_{mn}}\int\Delta^{2}(q_{1},\ldots,q_{m})\prod_{k=1}^{m}\left(e^{-q_{k}}q_{k}^{n-m}dq_{k}\sum_{a=0}^{\infty}\frac{\Gamma(\lambda+a)}{a!\Gamma(\lambda)}q_{k}^{a}x^{a}\right)\label{eq:F-integral}
\end{equation}
for positive integers $m$, $n$, $\lambda$ and $x$ satisfying $m\le n$,
where $\Delta(q_{1},\ldots,q_{m})$ is the Vandermonde determinant,
the integral is over the range $0\le q_{k}<\infty$ for all $1\le k\le m$,
\[
\Lambda_{mn}=\int\Delta^{2}(q_{1},\ldots,q_{m})\prod_{k=1}^{m}e^{-q_{k}}q_{k}^{n-m}dq_{k}.
\]
\end{thm}
\begin{proof}
From (\ref{eq:F}) and (\ref{eq:Sigma}) we have that
\begin{equation}
F(m,n,\lambda;x)=\sum_{N=0}^{\infty}\frac{\lambda^{N}}{N!}\sum_{r_{1}=1}^{\infty}\frac{x^{r_{1}}}{r_{1}}\ldots\sum_{r_{N}=1}^{\infty}\frac{x^{r_{N}}}{r_{N}}P_{r_{1}\ldots r_{N}}(m,n).\label{eq:F-expand}
\end{equation}
If we then substitute (\ref{eq:P_mult}) into this, we get that, when
$m\le n$,
\begin{eqnarray}
F(m,n,\lambda;x) & = & \frac{1}{\Lambda_{mn}}\sum_{N=0}^{\infty}\frac{\lambda^{N}}{N!}\sum_{r_{1}=1}^{\infty}\frac{x^{r_{1}}}{r_{1}}\ldots\sum_{r_{N}=1}^{\infty}\frac{x^{r_{N}}}{r_{N}}\int\Delta^{2}(q_{1},\ldots,q_{m})\nonumber \\
 &  & \qquad\times\prod_{k=1}^{m}(e^{-q_{k}}q_{k}^{n-m}dq_{k})\prod_{i=1}^{N}\sum_{j=1}^{m}q_{j}^{r_{i}}\nonumber \\
 & = & \frac{1}{\Lambda_{mn}}\sum_{N=0}^{\infty}\frac{\lambda^{N}}{N!}\int\Delta^{2}(q_{1},\ldots,q_{m})\prod_{k=1}^{m}(e^{-q_{k}}q_{k}^{n-m}dq_{k})\nonumber \\
 &  & \qquad\times\prod_{i=1}^{N}\sum_{j=1}^{m}\sum_{r_{i}=1}^{\infty}\frac{q_{j}^{r_{i}}x^{r_{i}}}{r_{i}}\nonumber \\
 & = & \frac{1}{\Lambda_{mn}}\int\Delta^{2}(q_{1},\ldots,q_{m})\prod_{k=1}^{m}(e^{-q_{k}}q_{k}^{n-m}dq_{k})\nonumber \\
 &  & \qquad\times\sum_{N=0}^{\infty}\frac{\lambda^{N}}{N!}\left(\sum_{j=1}^{m}\sum_{r=1}^{\infty}\frac{q_{j}^{r}x^{r}}{r}\right)^{N}.\label{eq:F-nested}
\end{eqnarray}

This expression is divergent for any given non-zero $x$, as almost
all of the domain of integration has at least one $q_{j}$ such that
$|q_{j}x|>1$, making 
\[
\sum_{r=1}^{\infty}\frac{q_{j}^{r}x^{r}}{r}
\]
diverge. However, we are still able to make more progress by considering
(\ref{eq:F-nested}) as a formal power series in $x$ again. we have
the identity
\[
\sum_{N=0}^{\infty}\frac{\lambda^{N}}{N!}\left(\sum_{j=1}^{m}\sum_{r=1}^{\infty}\frac{q_{j}^{r}x^{r}}{r}\right)^{N}=\prod_{j=1}^{m}\sum_{a_{j}=0}^{\infty}\frac{\Gamma(\lambda+a_{j})}{a_{j}!\Gamma(\lambda)}q_{j}^{a_{j}}x^{a_{j}}
\]
for $\lambda>0$ (see Theorem \ref{thm:Nested-series} in Appendix
1), so we rewrite (\ref{eq:F-nested}) as
\[
F(m,n,\lambda;x)=\frac{1}{\Lambda_{mn}}\int\Delta^{2}(q_{1},\ldots,q_{m})\prod_{k=1}^{m}e^{-q_{k}}q_{k}^{n-m}dq_{k}\sum_{a_{k}=0}^{\infty}\frac{\Gamma(\lambda+a_{k})}{a_{k}!\Gamma(\lambda)}q_{k}^{a_{k}}x^{a_{k}}.
\]

\end{proof}
\bigskip{}

This expression still bears similarities to expressions used in past
work \cite{Page1993,Foong1994,Sanchez-Ruiz1995,Sen1996,Dyer2014a}.
To evaluate the integral, we will use a method similar to that used
by Foong \cite{Foong1994}.

\bigskip{}

\begin{thm}
~

\begin{equation}
F(m,n,\lambda;x)=\sum_{a_{0}=0}^{\infty}\cdots\sum_{a_{m-1}=0}^{\infty}\prod_{0\le i<j<m}\left(\frac{a_{i}-a_{j}}{j-i}+1\right)\prod_{s=0}^{m-1}\frac{\Gamma(\lambda+a_{s})}{\Gamma(\lambda)}\frac{\Gamma(n-s+a_{s})}{\Gamma(n-s)}\frac{x^{a_{s}}}{a_{s}!}\label{eq:F-eval}
\end{equation}
for positive integers $m$, $n$, $\lambda$ and $x$ satisfying $m\le n$.\end{thm}
\begin{proof}
From Theorem \ref{thm:F-integral} we have
\[
F(m,n,\lambda;x)=\frac{1}{\Lambda_{mn}}\int\Delta^{2}(q_{1},\ldots,q_{m})\prod_{k=1}^{m}\left(e^{-q_{k}}q_{k}^{n-m}dq_{k}\sum_{a_{k}=0}^{\infty}\frac{\Gamma(\lambda+a_{k})}{a_{k}!\Gamma(\lambda)}q_{k}^{a_{k}}x^{a_{k}}\right).
\]
As in \cite{Foong1994}, we multiply the integrand by a ``damping
factor'' $\exp(-\sum_{k=1}^{m}\epsilon_{k}q_{k})$, and then replace
the $q_{i}$ in the Vandermonde discriminant $\Delta^{2}(q_{1},\ldots,q_{m})$
by $D_{i}=-\partial/\partial\epsilon_{i}$:
\begin{eqnarray}
F(m,n,\lambda;x) & = & \lim_{\epsilon\rightarrow0}\frac{\Delta^{2}(D_{1},\ldots,D_{m})}{\Lambda_{mn}}\prod_{k=1}^{m}\int_{0}^{\infty}e^{-(1+\epsilon_{k})q_{k}}q_{k}^{n-m}dq_{k}\nonumber \\
 &  & \left.\qquad\times\sum_{a_{k}=0}^{\infty}\frac{\Gamma(\lambda+a_{k})}{a_{k}!\Gamma(\lambda)}q_{k}^{a_{k}}x^{a_{k}}\right|_{\epsilon_{i}=\epsilon}.\label{eq:F-D}
\end{eqnarray}
Foong then notes that 
\[
\Delta^{2}(D_{1},\ldots,D_{m})=|\mathcal{DD}^{T}|,
\]
where
\[
\mathcal{D}=\left[\begin{array}{cccc}
1 & 1 & \cdots & 1\\
D_{1} & D_{2} & \cdots & D_{m}\\
\vdots & \vdots & \ddots & \vdots\\
D_{1}^{m-1} & D_{2}^{m-1} & \cdots & D_{m}^{m-1}
\end{array}\right],
\]
and that
\[
|\mathcal{DD}^{T}|\, f(\{\epsilon_{i}\})|_{\epsilon_{i}=\epsilon}=m!|\mathcal{D}|D_{2}D_{3}^{2}\cdots D_{m}^{m-1}f(\{\epsilon_{i}\})|_{\epsilon_{i}=\epsilon}
\]
when $f(\{\epsilon_{i}\})$ is a symmetric function of $\epsilon_{i}$
\cite{Foong1994}. Given this, we rewrite \ref{eq:F-D} as
\begin{eqnarray}
F(m,n,\lambda;x) & = & \lim_{\epsilon\rightarrow0}\frac{m!}{\Lambda_{mn}}|\mathcal{D}|\prod_{k=1}^{m}\int_{0}^{\infty}D_{k}^{k-1}e^{-(1+\epsilon_{k})q_{k}}q_{k}^{n-m}dq_{k}\nonumber \\
 &  & \left.\qquad\times\sum_{a_{k}=0}^{\infty}\frac{\Gamma(\lambda+a_{k})}{a_{k}!\Gamma(\lambda)}q_{k}^{a_{k}}x^{a_{k}}\right|_{\epsilon_{i}=\epsilon}\nonumber \\
 & = & \lim_{\epsilon\rightarrow0}\frac{m!}{\Lambda_{mn}}|\mathcal{D}|\prod_{k=1}^{m}\int_{0}^{\infty}e^{-(1+\epsilon_{k})q_{k}}q_{k}^{n-m+k-1}dq_{k}\nonumber \\
 &  & \left.\qquad\times\sum_{a_{k}=0}^{\infty}\frac{\Gamma(\lambda+a_{k})}{a_{k}!\Gamma(\lambda)}q_{k}^{a_{k}}x^{a_{k}}\right|_{\epsilon_{i}=\epsilon}\nonumber \\
 & = & \lim_{\epsilon\rightarrow0}\left.\frac{m!}{\Lambda_{mn}}|\mathcal{D}|\prod_{k=1}^{m}\sum_{a_{k}=0}^{\infty}\frac{\Gamma(\lambda+a_{k})}{a_{k}!\Gamma(\lambda)}\frac{\Gamma(n-m+k+a_{k})}{(1+\epsilon_{k})^{n-m+k+a_{k}}}x^{a_{k}}\right|_{\epsilon_{i}=\epsilon}\nonumber \\
 & = & \frac{m!}{\Lambda_{mn}}\sum_{a_{1}=0}^{\infty}\frac{\Gamma(\lambda+a_{1})}{a_{1}!\Gamma(\lambda)}x^{a_{1}}\cdots\sum_{a_{m}=0}^{\infty}\frac{\Gamma(\lambda+a_{m})}{a_{m}!\Gamma(\lambda)}x^{a_{m}}\nonumber \\
 &  & \qquad\times\lim_{\epsilon\rightarrow0}\left.|\mathcal{D}|\prod_{k=1}^{m}\frac{\Gamma(n-m+k+a_{k})}{(1+\epsilon_{k})^{n-m+k+a_{k}}}x^{a_{k}}\right|_{\epsilon_{i}=\epsilon}.\label{eq:F-det}
\end{eqnarray}

Let
\[
Q=\lim_{\epsilon\rightarrow0}\left.|\mathcal{D}|\prod_{k=1}^{m}\frac{\Gamma(n-m+k+a_{k})}{(1+\epsilon_{k})^{n-m+k+a_{k}}}\right|_{\epsilon_{i}=\epsilon}.
\]
Expanding the determinant out explicitly in terms of the Levi-Civita
symbol, we get
\begin{eqnarray*}
Q & = & \lim_{\epsilon\rightarrow0}\left.\varepsilon_{i_{1}\ldots i_{m}}\prod_{k=1}^{m}D_{k}^{i_{k}-1}\frac{\Gamma(n-m+k+a_{k})}{(1+\epsilon_{k})^{n-m+k+a_{k}}}\right|_{\epsilon_{i}=\epsilon}\\
 & = & \lim_{\epsilon\rightarrow0}\left.\varepsilon_{i_{1}\ldots i_{m}}\prod_{k=1}^{m}\frac{\Gamma(n-m+k+i_{k}-1+a_{k})}{(1+\epsilon_{k})^{n-m+k+i_{k}-1+a_{k}}}\right|_{\epsilon_{i}=\epsilon}\\
 & = & \varepsilon_{i_{1}\ldots i_{m}}\Gamma(n-m+k+i_{k}-1+a_{k})\\
 & = & \left|\begin{array}{cccc}
\Gamma(n-m+1+a_{1}) & \Gamma(n-m+2+a_{2}) & \cdots & \Gamma(n+a_{m})\\
\Gamma(n-m+2+a_{1}) & \Gamma(n-m+3+a_{2}) & \cdots & \Gamma(n+1+a_{m})\\
\Gamma(n-m+3+a_{1}) & \Gamma(n-m+4+a_{2}) & \cdots & \Gamma(n+2+a_{m})\\
\vdots & \vdots & \ddots & \vdots\\
\Gamma(n+a_{1}) & \Gamma(n+1+a_{2}) & \cdots & \Gamma(n+m-1+a_{m})
\end{array}\right|.
\end{eqnarray*}
We then simplify this determinant by a process of subtracting multiples
of different rows of the matrix from each other as follows:
\begin{enumerate}
\item Subtract $(n-m)$ times the first row from the second, $(n-m+1)$
times the second from the third etc. to give
\[
Q=\left|\tiny{\begin{array}{cccc}
\Gamma(n-m+1+a_{1}) & \Gamma(n-m+2+a_{2}) & \cdots & \Gamma(n+a_{m})\\
(a_{1}+1)\Gamma(n-m+1+a_{1}) & (a_{2}+2)\Gamma(n-m+2+a_{2}) & \cdots & (a_{m}+m)\Gamma(n+a_{m})\\
(a_{1}+1)\Gamma(n-m+2+a_{1}) & (a_{2}+2)\Gamma(n-m+3+a_{2}) & \cdots & (a_{m}+m)\Gamma(n+1+a_{m})\\
\vdots & \vdots & \ddots & \vdots\\
(a_{1}+1)\Gamma(n-1+a_{1}) & (a_{2}+2)\Gamma(n+a_{2}) & \cdots & (a_{m}+m)\Gamma(n+m-2+a_{m})
\end{array}}\right|.
\]

\item Subtract $(n-m)$ times the second row from the third, $(n-m+1)$
times the third row from the fourth etc. to give
\[
Q=\left|\tiny{\begin{array}{cccc}
\Gamma(n-m+1+a_{1}) & \Gamma(n-m+2+a_{2}) & \cdots & \Gamma(n+a_{m})\\
(a_{1}+1)\Gamma(n-m+1+a_{1}) & (a_{2}+2)\Gamma(n-m+2+a_{2}) & \cdots & (a_{m}+m)\Gamma(n+a_{m})\\
(a_{1}+1)^{2}\Gamma(n-m+1+a_{1}) & (a_{2}+2)^{2}\Gamma(n-m+2+a_{2}) & \cdots & (a_{m}+m)^{2}\Gamma(n+a_{m})\\
\vdots & \vdots & \ddots & \vdots\\
(a_{1}+1)^{2}\Gamma(n-2+a_{1}) & (a_{2}+2)^{2}\Gamma(n-1+a_{2}) & \cdots & (a_{m}+m)^{2}\Gamma(n+m-3+a_{m})
\end{array}}\right|.
\]

\item Continue to repeat this process, starting a row further down each
time, until
\begin{eqnarray*}
Q & = & \left|\tiny{\begin{array}{cccc}
\Gamma(n-m+1+a_{1}) & \Gamma(n-m+2+a_{2}) & \cdots & \Gamma(n+a_{m})\\
(a_{1}+1)\Gamma(n-m+1+a_{1}) & (a_{2}+2)\Gamma(n-m+2+a_{2}) & \cdots & (a_{m}+m)\Gamma(n+a_{m})\\
(a_{1}+1)^{2}\Gamma(n-m+1+a_{1}) & (a_{2}+2)^{2}\Gamma(n-m+2+a_{2}) & \cdots & (a_{m}+m)^{2}\Gamma(n+a_{m})\\
\vdots & \vdots & \ddots & \vdots\\
(a_{1}+1)^{m-1}\Gamma(n-m+1+a_{1}) & (a_{2}+2)^{m-1}\Gamma(n-m+2+a_{2}) & \cdots & (a_{m}+m)^{m-1}\Gamma(n+a_{m})
\end{array}}\right|\\
 & = & \Delta(a_{1}+1,\ldots,a_{m}+m)\prod_{k=1}^{m}\Gamma(n-m+a_{k}+k)\\
 & = & \prod_{0<i<j\le m}(a_{j}-a_{i}+j-i)\prod_{k=1}^{m}\Gamma(n-m+a_{k}+k).
\end{eqnarray*}

\end{enumerate}
We then substitute this expression into (\ref{eq:F-det}):
\begin{eqnarray*}
F(m,n,\lambda;x) & = & \frac{m!}{\Lambda_{mn}}\sum_{a_{1}=0}^{\infty}\frac{\Gamma(\lambda+a_{1})}{a_{1}!\Gamma(\lambda)}x^{a_{1}}\cdots\sum_{a_{m}=0}^{\infty}\frac{\Gamma(\lambda+a_{m})}{a_{m}!\Gamma(\lambda)}x^{a_{m}}\\
 &  & \qquad\times\lim_{\epsilon\rightarrow0}\left.|\mathcal{D}|\prod_{k=1}^{m}\frac{\Gamma(n-m+k+a_{k})}{(1+\epsilon_{k})^{n-m+k+a_{k}}}\right|_{\epsilon_{i}=\epsilon}\\
 & = & \frac{m!}{\Lambda_{mn}}\sum_{a_{1}=0}^{\infty}\frac{\Gamma(\lambda+a_{1})}{a_{1}!\Gamma(\lambda)}x^{a_{1}}\cdots\sum_{a_{m}=0}^{\infty}\frac{\Gamma(\lambda+a_{m})}{a_{m}!\Gamma(\lambda)}x^{a_{m}}\\
 &  & \qquad\times\prod_{0<i<j\le m}(a_{j}-a_{i}+j-i)\prod_{k=1}^{m}\Gamma(n-m+a_{k}+k),
\end{eqnarray*}
and simplify this by making the substitutions $i\rightarrow m-i$,
$j\rightarrow m-j$, $k\rightarrow m-k$ and $a_{s}\rightarrow a_{m-s}$
such that
\begin{eqnarray*}
F(m,n,\lambda;x) & = & \frac{m!}{\Lambda_{mn}}\sum_{a_{0}=0}^{\infty}\frac{\Gamma(\lambda+a_{0})}{a_{0}!\Gamma(\lambda)}x^{a_{0}}\cdots\sum_{a_{m-1}=0}^{\infty}\frac{\Gamma(\lambda+a_{m-1})}{a_{m-1}!\Gamma(\lambda)}x^{a_{m-1}}\\
 &  & \qquad\times\prod_{0\le i<j<m}(a_{i}-a_{j}+j-i)\prod_{s=0}^{m-1}\Gamma(n-s+a_{s}).
\end{eqnarray*}
We know from (\ref{eq:F_0}) that $F(m,n,\lambda;0)=1$, which means
we can now fix the value of the normalisation constant $\Lambda_{mn}$,
as
\[
F(m,n,\lambda;0)=\frac{m!}{\Lambda_{mn}}\prod_{0\le i<j<m}(j-i)\prod_{s=0}^{m-1}\Gamma(n-s)=1.
\]
Therefore,
\[
F(m,n,\lambda;x)=\sum_{a_{0}=0}^{\infty}\cdots\sum_{a_{m-1}=0}^{\infty}\prod_{0\le i<j<m}\left(\frac{a_{i}-a_{j}}{j-i}+1\right)\prod_{s=0}^{m-1}\frac{\Gamma(\lambda+a_{s})}{\Gamma(\lambda)}\frac{\Gamma(n-s+a_{s})}{\Gamma(n-s)}\frac{x^{a_{s}}}{a_{s}!}.
\]

\end{proof}
\medskip{}

This expression can now be used to evaluate any given $F_{r}$, by
summing only over the cases $\{a_{0},\ldots,a_{m-1}\}$ for which
$a_{0}+\ldots+a_{m-1}=r$. Unlike the closed-form expressions we derived
previously in \cite{Dyer2014a}, however, this expression cannot be
used directly to find polynomial expansions for $F_{r}(m,n,\lambda)$,
due to the dependence on $m$ of the number of summations and the
ranges of the products.

However, we know that each $H_{r}$ must be a symmetric polynomial
(Theorem \ref{thm:symmetric}) of order at $r$ in each of its parameters
(a hypermap with $r$ darts can have a most $r$ edges), and that
$H_{r}(m,n,\lambda)=0$ if any of its parameters are zero (all hypermaps
must have at least one each of vertices, edges and faces). Given that
$F_{0}(m,n,\lambda)=1$, it follows from (\ref{eq:generate-recurse})
that $F_{r}(m,n,\lambda)$ for any $r>0$ is also symmetric, order
$r$ in each parameter, and zero when $m$, $n$ or $\lambda$ are
zero. Therefore, we can compute the polynomial coefficients of $H_{r}$
(and therefore enumerate rooted hypermaps) by evaluating $F_{r}(m,n,\lambda)$
-- and by extension $H_{r}(m,n,\lambda)$ -- using (\ref{eq:F-eval})
and (\ref{eq:generate-recurse}) at all $1\le m\le n\le\lambda\le r$
and using polynomial interpolation.

We used this method to compute the coefficients of $H_{r}$ for all
$1\le r\le13$. Some of the output is given in Appendix 2, and the
results agree exactly with past computations, in particular Walsh's
enumeration of all rooted hypermaps up to $r=12$ \cite{Walsh2012}.
Running on a 2012 Dell XPS 12 these calculations took 107 minutes,
in comparison to the few days taken by Walsh's algorithm.

\subsection{Special cases}

While no simple closed-form polynomial expressions are available for
$H_{r}(m,n,\lambda)$, there are a few special cases in which we can
get more useful results.

Consider the function 
\begin{equation}
H_{r}(1,m,n)=\sum_{v,e,f}H_{vefr}m^{e}n^{f}.\label{eq:P_2}
\end{equation}
This is the generating function for enumerating rooted hypermaps with
$r$ darts by number of edges and faces (with all possible numbers
of vertices summed over). By the symmetry of $H_{r}$, (\ref{eq:P_2})
could also be used to enumerate by number of vertices and edges, summing
over all numbers of faces etc.

\bigskip{}

\begin{thm}
For all $r>0$,
\begin{equation}
H_{r}(1,m,n)=\frac{1}{(r-1)!}\frac{\Gamma(m+r)}{\Gamma(m)}\frac{\Gamma(n+r)}{\Gamma(n)}-\sum_{k=1}^{r-1}\frac{1}{k!}\frac{\Gamma(m+k)}{\Gamma(m)}\frac{\Gamma(n+k)}{\Gamma(n)}H_{r-k}(1,m,n).\label{eq:P_2_eval}
\end{equation}
\end{thm}
\begin{proof}
From (\ref{eq:F-eval}) we have that
\[
F(1,m,n;x)=\sum_{a=0}^{\infty}\frac{\Gamma(m+a)}{\Gamma(m)}\frac{\Gamma(n+a)}{\Gamma(n)}\frac{x^{a}}{a!},
\]
so
\[
F_{r}(1,m,n)=\frac{1}{r!}\frac{\Gamma(m+r)}{\Gamma(m)}\frac{\Gamma(n+r)}{\Gamma(n)}.
\]
We substitute this into (\ref{eq:generate-recurse}) and rearrange
to get
\begin{eqnarray*}
H_{r}(1,m,n) & = & rF_{r}(1,m,n)-\sum_{k=1}^{r-1}F_{k}(1,m,n)H_{r-k}(1,m,n)\\
 & = & \frac{1}{(r-1)!}\frac{\Gamma(m+r)}{\Gamma(m)}\frac{\Gamma(n+r)}{\Gamma(n)}\\
 &  & -\sum_{k=1}^{r-1}\frac{1}{k!}\frac{\Gamma(m+k)}{\Gamma(m)}\frac{\Gamma(n+k)}{\Gamma(n)}H_{r-k}(1,m,n).
\end{eqnarray*}

\bigskip{}

In contrast to expressions such as (\ref{eq:F-integral}) and (\ref{eq:F-eval}),
this expression obviously gives rise to symmetric polynomial functions.

Given (\ref{eq:P_2_eval}), the following two results follow trivially:

\bigskip{}
\end{proof}
\begin{cor}
For all $r>0$,
\[
H_{r}(1,1,m)=r\frac{\Gamma(m+r)}{\Gamma(m)}-\sum_{k=1}^{r-1}\frac{\Gamma(m+k)}{\Gamma(m)}H_{r-k}(1,1,m).
\]

\end{cor}

\begin{cor}
For all $r>0$,
\[
H_{r}(1,1,1)=r\cdot r!-\sum_{k=1}^{r-1}k!H_{r-k}(1,1,1).
\]

\end{cor}
\bigskip{}

The second in particular allows us to count how many rooted hypermaps
there are in total with $r$ darts. The first few values are 1, 3,
13, 71, 461...

\section{Conclusions}

We have demonstrated a method for computing generating functions to
enumerate rooted hypermaps by number of vertices, edges and faces
for any given number of darts. This is an extension of previous work
where we derived closed form generating functions for counting enumerating
rooted hypermaps with one face \cite{Dyer2014a}, but in contrast
to that case the method shown here defines the generating function
$H_{r}$ for $r$ darts recursively in terms of $H_{1},\ldots,H_{r-1}$,
and it only allows $H_{r}$ to be evaluated numerically, not expanded
directly as a polynomial. We were still able to obtain a polynomial
expansion, however, by using polynomial interpolation.

This work is a further demonstration of the use of matrix integration
as a tool for finding generating functions for enumerating sets of
combinatoric objects. It specifically demonstrates the link, first
discussed in \cite{Dyer2014a}, between rooted hypermaps and the ensemble
of reduced density operators on random states of a bipartite quantum
system.

We also discussed a number of related results. First we showed the
symmetry of the generating functions $H_{r}$, arising from the symmetry
of 3-constellations, and used this to speed up computation of $H_{r}$
by reducing the range over which $H_{r}$ needed to be evaluated to
fix the polynomial expansion. Then we looked at cases where one or
more parameters in $H_{r}$ were set to unity, giving generating functions
for enumerating larger sets of rooted hypermaps (such as all those
with $r$ darts and $f$ faces, summing over all possible numbers
of edges and vertices). In particular, this allowed us to easily count
all rooted hypermaps with $r$ darts and any number of edges, vertices
and faces.

\section{Acknowledgements}

The work in this paper was supported by an EPSRC research studentship
at the University of York.

\section*{Appendix 1}
\begin{thm}
\label{thm:Nested-series}For positive integer $m$ and $\lambda>0$,
the formal power series
\[
\sum_{N=0}^{\infty}\frac{\lambda^{N}}{N!}\left(\sum_{j=1}^{m}\sum_{r=1}^{\infty}\frac{q_{j}^{r}x^{r}}{r}\right)^{N}=\prod_{j=1}^{m}\sum_{a_{j}=0}^{\infty}\frac{\Gamma(\lambda+a_{j})}{a_{j}!\Gamma(\lambda)}q_{j}^{a_{j}}x^{a_{j}},
\]
where $q_{j}$ are components of an $m$-dimensional real vector.\end{thm}
\begin{proof}
Let 
\begin{equation}
L_{\lambda,q}(x)=\sum_{N=0}^{\infty}\frac{\lambda^{N}}{N!}\left(\sum_{j=1}^{m}\sum_{r=1}^{\infty}\frac{q_{j}^{r}x^{r}}{r}\right)^{N}.\label{eq:series-1}
\end{equation}
For any given positive integer $a$, we see by inspection that this
series contains only a finite number of terms of order $x^{a}$, as
such terms can only come from cases where $1\le N\le a$. In addition,
there is only one constant term: the $N=0$ case which equals unity.
Therefore, $L_{\lambda,q}(x)$ can be written in the form
\begin{equation}
L_{\lambda,q}(x)=\sum_{a=0}^{\infty}f_{a}(\lambda,q)x^{a}\label{eq:Taylor-general}
\end{equation}
where each $f_{a}(\lambda,q)$ is a polynomial in $\lambda$ and $q$.

$L_{\lambda,q}(x)$ converges when $(|q_{j}x|)<1$ for all $j$ to
\begin{eqnarray*}
L_{\lambda}(x) & = & \sum_{N=0}^{\infty}\frac{\lambda^{N}}{N!}\left(-\sum_{j=1}^{m}\ln(1-q_{j}x)\right)\\
 & = & \exp[-\lambda\sum_{j=1}^{m}\ln(1-q_{j}x)]\\
 & = & \prod_{j=1}^{m}\frac{1}{(1-q_{j}x)^{\lambda}}.
\end{eqnarray*}
This has a series expansion in $x$, also valid when $|q_{j}x|<1$
for all $j$, of
\begin{equation}
\prod_{j=1}^{m}\sum_{a_{j}=0}^{\infty}\frac{\Gamma(\lambda+a_{j})}{a_{j}!\Gamma(\lambda)}q_{j}^{a_{j}}x^{a_{j}}.\label{eq:series-2}
\end{equation}
This can also be rearranged into the form (\ref{eq:Taylor-general}).
(\ref{eq:series-1}) and (\ref{eq:series-2}) are therefore both Taylor
series with the same radius of convergence, and they are equal to
each other everywhere within it, so it follows from the uniqueness
of Taylor series expansions of smuooth functions that they are equivalent,
i.e.

\[
\sum_{N=0}^{\infty}\frac{\lambda^{N}}{N!}\left(\sum_{j=1}^{m}\sum_{r=1}^{\infty}\frac{q_{j}^{r}x^{r}}{r}\right)^{N}=\prod_{j=1}^{m}\sum_{a_{j}=0}^{\infty}\frac{\Gamma(\lambda+a_{j})}{a_{j}!\Gamma(\lambda)}q_{j}^{a_{j}}x^{a_{j}}.
\]

\end{proof}
\pagebreak{}

\section*{Appendix 2}

Numbers of rooted hypermaps with $v$ vertices, $e$ edges, $f$ faces
and $r$ darts, calculated by computing the generating functions $H_{r}$.
Only the cases with $v\le e\le f$ are given, as the rest follow from
the symmetry of $H_{r}$. The cases $1\le r\le7$ are included for
comparison with Walsh's previous computation \cite{Walsh2012}, with
all cases up to $r=12$ agreeing with his computation. The new case
$r=13$ is also shown.

\begin{minipage}[t]{0.5\textwidth}
\tiny

$\boldsymbol{r=1:}$\\
\begin{tabular}{ccc|c}
$v$ & $e$ & $f$ & $N$\tabularnewline
\hline 
1 & 1 & 1 & 1\tabularnewline
\end{tabular}

\medskip{}
$\boldsymbol{r=2:}$\\
\begin{tabular}{ccc|c}
$v$ & $e$ & $f$ & $N$\tabularnewline
\hline 
1 & 1 & 2 & 1\tabularnewline
\end{tabular}

\medskip{}
$\boldsymbol{r=3:}$\\
\begin{tabular}{ccc|c}
$v$ & $e$ & $f$ & $N$\tabularnewline
\hline 
1 & 1 & 1 & 1\tabularnewline
1 & 2 & 2 & 3\tabularnewline
1 & 1 & 3 & 1\tabularnewline
\end{tabular}

\medskip{}
$\boldsymbol{r=4:}$\\
\begin{tabular}{ccc|c}
$v$ & $e$ & $f$ & $N$\tabularnewline
\hline 
1 & 1 & 2 & 5\tabularnewline
2 & 2 & 2 & 17\tabularnewline
1 & 2 & 3 & 6\tabularnewline
1 & 1 & 4 & 1\tabularnewline
\end{tabular}

\medskip{}

$\boldsymbol{r=5:}$\\
\begin{tabular}{ccc|c}
$v$ & $e$ & $f$ & $N$\tabularnewline
\hline 
1 & 1 & 1 & 8\tabularnewline
1 & 2 & 2 & 40\tabularnewline
1 & 1 & 3 & 15\tabularnewline
2 & 2 & 3 & 55\tabularnewline
1 & 3 & 3 & 20\tabularnewline
1 & 2 & 4 & 10\tabularnewline
1 & 1 & 5 & 1\tabularnewline
\end{tabular}

\medskip{}

$\boldsymbol{r=6:}$\\
\begin{tabular}{ccc|c}
$v$ & $e$ & $f$ & $N$\tabularnewline
\hline 
1 & 1 & 2 & 84\tabularnewline
2 & 2 & 2 & 456\tabularnewline
1 & 2 & 3 & 175\tabularnewline
2 & 3 & 3 & 262\tabularnewline
1 & 1 & 4 & 35\tabularnewline
2 & 2 & 4 & 135\tabularnewline
1 & 3 & 4 & 50\tabularnewline
1 & 2 & 5 & 15\tabularnewline
1 & 1 & 6 & 1\tabularnewline
\end{tabular}

\medskip{}

$\boldsymbol{r=7:}$\\
\begin{tabular}{ccc|c}
$v$ & $e$ & $f$ & $N$\tabularnewline
\hline 
1 & 1 & 1 & 180\tabularnewline
1 & 2 & 2 & 1183\tabularnewline
1 & 1 & 3 & 469\tabularnewline
2 & 2 & 3 & 2695\tabularnewline
1 & 3 & 3 & 1050\tabularnewline
3 & 3 & 3 & 1694\tabularnewline
1 & 2 & 4 & 560\tabularnewline
2 & 3 & 4 & 889\tabularnewline
1 & 4 & 4 & 175\tabularnewline
1 & 1 & 5 & 70\tabularnewline
2 & 2 & 5 & 280\tabularnewline
1 & 3 & 5 & 105\tabularnewline
1 & 2 & 6 & 21\tabularnewline
1 & 1 & 7 & 1\tabularnewline
\end{tabular}

\normalsize
\end{minipage}
\begin{minipage}[t]{0.5\textwidth}
\tiny

$\boldsymbol{r=13:}$\\
\begin{tabular}{ccc|c}
$v$ & $e$ & $f$ & $N$\tabularnewline
\hline 
1 & 1 & 1 & 68428800\tabularnewline
1 & 2 & 2 & 686597184\tabularnewline
1 & 1 & 3 & 292271616\tabularnewline
2 & 2 & 3 & 2820651496\tabularnewline
1 & 3 & 3 & 1194737544\tabularnewline
3 & 3 & 3 & 4623070842\tabularnewline
1 & 2 & 4 & 687238552\tabularnewline
2 & 3 & 4 & 2646424729\tabularnewline
1 & 4 & 4 & 636184120\tabularnewline
3 & 4 & 4 & 2239280420\tabularnewline
1 & 1 & 5 & 109425316\tabularnewline
2 & 2 & 5 & 988043771\tabularnewline
1 & 3 & 5 & 414918075\tabularnewline
3 & 3 & 5 & 1453414846\tabularnewline
2 & 4 & 5 & 824962502\tabularnewline
4 & 4 & 5 & 582408775\tabularnewline
1 & 5 & 5 & 125855730\tabularnewline
3 & 5 & 5 & 374805834\tabularnewline
5 & 5 & 5 & 64013222\tabularnewline
1 & 2 & 6 & 108452916\tabularnewline
2 & 3 & 6 & 374127663\tabularnewline
1 & 4 & 6 & 87933846\tabularnewline
3 & 4 & 6 & 260619268\tabularnewline
2 & 5 & 6 & 93880696\tabularnewline
4 & 5 & 6 & 44136820\tabularnewline
1 & 6 & 6 & 9513504\tabularnewline
3 & 6 & 6 & 19315114\tabularnewline
1 & 1 & 7 & 8691683\tabularnewline
2 & 2 & 7 & 70367479\tabularnewline
1 & 3 & 7 & 29135106\tabularnewline
3 & 3 & 7 & 85050784\tabularnewline
2 & 4 & 7 & 47604648\tabularnewline
4 & 4 & 7 & 22089600\tabularnewline
1 & 5 & 7 & 6936930\tabularnewline
3 & 5 & 7 & 14019928\tabularnewline
2 & 6 & 7 & 3356522\tabularnewline
1 & 7 & 7 & 226512\tabularnewline
1 & 2 & 8 & 4114110\tabularnewline
2 & 3 & 8 & 11674663\tabularnewline
1 & 4 & 8 & 2642640\tabularnewline
3 & 4 & 8 & 5264545\tabularnewline
2 & 5 & 8 & 1827683\tabularnewline
1 & 6 & 8 & 169884\tabularnewline
1 & 1 & 9 & 183183\tabularnewline
2 & 2 & 9 & 1225653\tabularnewline
1 & 3 & 9 & 495495\tabularnewline
3 & 3 & 9 & 960960\tabularnewline
2 & 4 & 9 & 525525\tabularnewline
1 & 5 & 9 & 70785\tabularnewline
1 & 2 & 10 & 40040\tabularnewline
2 & 3 & 10 & 74217\tabularnewline
1 & 4 & 10 & 15730\tabularnewline
1 & 1 & 11 & 1001\tabularnewline
2 & 2 & 11 & 4433\tabularnewline
1 & 3 & 11 & 1716\tabularnewline
1 & 2 & 12 & 78\tabularnewline
1 & 1 & 13 & 1\tabularnewline
\end{tabular}

\normalsize
\end{minipage}

\bibliographystyle{plain}
\bibliography{Hypermaps2}

\end{document}